\documentclass[12pt]{amsart}
\usepackage{amssymb, eucal, amsfonts, amsmath, amssymb, xypic,latexsym, amscd}

\def\k{{\Bbbk}}
\def\L{{\mathcal L}}
\def\g{{\mathfrak g}}
\def\A{{\mathbb A}}
\def\Z{{\mathbb Z}}
\def\h{{\mathfrak h}}

\def\t{{\mathfrak t}}

\def\NN{{\mathcal N}}

\def\OO{{\mathcal O}}

\def\L{{\mathcal L}}

\def\z{{\mathfrak z}}
\def\m{{\mathfrak m}}

\def\n{{\mathfrak n}}

\def\Z{{\mathbb Z}}
\def\End{\mathop{\fam0 End}}

\def\Lie{\mathop{\fam0 Lie}}

\def\ker{\mathrm{Ker}\,}
\def\Der{\mathop{\fam0 Der}}

\def\ad{\mathrm{ad\,}}

\theoremstyle{plain}
\newtheorem{theorem}{Theorem}
\newtheorem{corollary}{Corollary}
\newtheorem{prop}{Proposition}
\newtheorem{lemma}{Lemma}

\theoremstyle{definition}

\theoremstyle{remark}
\newtheorem{rem}{Remark}

\def\subtitle#1. {{\medskip\bf#1\par\nobreak\smallskip}}
\def\proclaim#1. {\medbreak\bgroup\noindent\bf#1. \it}

\def\endproclaim{\egroup
\ifdim\lastskip<\medskipamount\removelastskip\medskip\fi}
\newcount
=0
\def\citedef#1 {\advance\citation by1
  \expandafter\edef\csname#1\endcsname{{\the\citation}}
  \checkendcitedef}
\def\checkendcitedef#1{\ifx#1\endcitedef\else\citedef#1\fi}
\def\cite#1{\csname#1\endcsname}
\citedef  Bl69 Di11
Jac43  K63 MR  PPY07 P86 P89  P88 P91 P03 Sh94 Sk02
Sk14 St87
St1
\endcitedef
\newtoks\nextauth
\newif\iffirstauth
\def\checkendauth#1{\ifx\endauth#1
        \iffirstauth\the\nextauth
        \else{} and \the\nextauth\fi,
    \else\iffirstauth\the\nextauth\firstauthfalse
        \else, \the\nextauth\fi
        \expandafter\auth\expandafter#1\fi}
\def\auth#1 #2 {\nextauth={#1 #2}\checkendauth}
\newif\ifinbook
\newif\ifbookref
\def\nextref#1 {\bookreffalse\inbookfalse
    \bibitem[\cite{#1}]{}
    \firstauthtrue
    \ignorespaces}
\def\paper#1{{\it#1,}}

\def\book#1{\bookreftrue{\it#1,}}
\def\journal#1{#1\ifinbook,\fi}
\def\bookseries#1{#1,}
\def\Vol#1{\ifbookref Vol. #1,\else\ifinbook Vol. #1,\else{\bf#1}\fi\fi
    \space\ignorespaces}
\def\nombre#1{no. #1}
\def\publisher#1{#1,}
\def\Year#1{\ifbookref #1.\else\ifinbook #1,\else(#1)\fi\fi
    \space\ignorespaces}
\def\Pages#1{\ifinbook pp. #1.\else #1.\fi}
\begin{document}
\title{Regular derivations of truncated polynomial rings}
\author{Alexander Premet}
\thanks{\nonumber{\it Mathematics Subject Classification} (2000 {\it revision}).
Primary 17B50. Secondary 13A50.}
\address{School of Mathematics, University of Manchester, Oxford Road,
M13 9PL, UK} \email{Alexander.Premet@manchester.ac.uk}
\begin{abstract}
\noindent
Let $\k$ be an algebraically closed field $\k$ of characteristic $p>2$ and let
$\OO_n=\k[X_1,\ldots, X_n]/(X_1^p,\ldots, X_n^p)$, a local $\k$-algebra of dimension $p^n$ over $\k$. If $p=3$ we assume that $n>1$ and
impose no restrictions on $n$ for $p>3$. Let $\L$ be the Lie algebra of all derivations of $\OO_n$, a restricted simple Lie algebra of Cartan type $W_n$, and denote by $G$ be the automorphism group of $\L$.
By [\cite{P91}], the invariant ring $\k[\L]^G$ is freely generated by $n$ homogeneous polynomial functions $\psi_0,\ldots, \psi_{n-1}$ and a version of Chevalley's Restriction Theorem holds for $\L$. Moreover, the majority of classical results of Kostant on the adjoint action of a complex reductive group on its Lie algebra hold for the action of $G$ on $\L$. In particular, each fibre of the map $\psi\colon\L\to \A^n$ sending any $x\in\L$ to $(\psi_0(x),\ldots,\psi_{n-1}(x))\in\A^n$ is an irreducible complete intersection in $\L$ and contains an open $G$-orbit.
However, it is also proved in [\cite{P91}] that the zero fibre of $\psi$ is not a normal variety.
In this paper, we complete the picture
by showing that Kostant's differential criterion for regularity holds in $\L$ and we prove that a
fibre of $\psi$ is normal if and only if it consists of regular semisimple elements of $\L$.

\end{abstract}
\maketitle

\bigskip

\bigskip

\bigskip

\hfill {\it To my friend Helmut Strade with admiration\ \ }

\bigskip

\section{\bf Introduction}\label{intro}
\subsection{}
Let $\OO_n$ be the truncated polynomial ring $\k[X_1,\ldots, X_n]/(X_1^p,\ldots, X_n^p)$ over an algebraically closed field $\k$ of characteristic $p>2$ and $\L=\Der(\OO_n)$. If $(p,n)=(3,1)$ then $\L\cong \mathfrak{sl}_2(\k)$. We therefore exclude this case and assume that $(p,n)\ne (3,1)$.
Let $G$ denote the automorphism group of $\OO_n$. It is well known that $G$ is a connected algebraic $\k$-group
of dimension $n(p^n-1)$ and
$G/R_u(G)\cong \mathrm{GL}_n(\k)$.
Furthermore, under our assumptions on $(p,n)$ any automorphism of the Lie algebra $\L$ is induced  by a unique automorphism
of the local $\k$-algebra $\OO_n$ so that ${\rm Aut}(\L)\cong G$ as algebraic $\k$-groups.
We denote by $x_i$ the image of $X_i$ in $\OO_n$.
It is straightforward to see that $\L$ is a free $\OO_n$-module of rank $n$ with basis consisting of partial derivatives $\partial_i=\frac{\partial_i}{\partial x_i}$ where $1\le i\le n$
\subsection{}
Being a full derivation algebra, $\L$ carries a natural $p$th power map $x\mapsto x^p$ equivariant under the action of $G$. One knows that all Cartan subalgebras of the restricted Lie algebra $\L$ are toral and have dimension $n$. There are precisely
$n+1$ conjugacy classes of such subalgebras under the action of $G$. As a canonical representative of the $k$th conjugacy class one usually takes the torus
$$\t_k:=\k(x_1\partial_1)\oplus\cdots\oplus\k(x_k\partial_k)
\oplus(1+x_{k+1})\partial_{k+1}\oplus\cdots\oplus
\k(1+x_n)\partial_n, \qquad 0\le k\le n.$$

Since $\L$ contains an $n$-dimensional, self-centralising torus, some general results proved in [\cite{P89}] show that there exist
algebraically independent, homogeneous polynomial functions $\psi_0,\ldots,\psi_{n-1}\in\k[\L]^G$ with $\deg\psi_i=p^n-p^i$ such that
$$x^{p^n}+\textstyle{\sum}_{i=0}^{n-1}\,
\psi_i(x)x^{p^i}=0\qquad\quad\  (\forall\,x\in\L);$$ see Subsection~2.2 for detail.
The $G$-saturation of $\t_0$ is known to contain a nonempty Zariski open subset of $\L$. So the restriction map $\k[\L]\to \k[\t_0]$
induces an injection $j\colon\, \k[\L]^G \to \k[\t_0]^{N_G(\t_0)}$ of invariant rings.
By [\cite{P91}], the group $N_G(\t_0)$ is isomorphic to $\mathrm{GL}_n(\mathbb{F}_p)$ and acts on $\t_0$ faithfully. In conjunction with classical results of Dickson [\cite{Di11}] this implies that $j$ is surjective and hence
$\k[\L]^G\cong\, \k[\psi_0,\ldots,\psi_{n-1}]$ as $\k$-algebras.
\subsection{} Let $\psi\colon\,\L\to\mathbb{A}^n$ denote the map sending any $x\in\L$
to $\big(\psi_0(x),\ldots,\psi_{n-1}(x)\big)\in\mathbb{A}^n$. By [\cite{P91}],
the morphism $\psi$ is flat, surjective, and for any
$\mathbf{\eta}=(\eta_0,\ldots,\eta_{n-1})\in\mathbb{A}^n$ the fibre $P_{\mathbf{\eta}}:=\psi^{-1}(\eta)$ is an irreducible complete intersection in $\L$ whose defining ideal in $\k[\L]$ is generated by $\psi_0-\eta_0,\ldots,\psi_{n-1}-\eta_{n-1}$. Each
$P_{\mathbf{\eta}}$ contains a unique open $G$-orbit,
denoted $P_{\mathbf{\eta}}^{\,\circ}$, which consists of all elements  of $P_{\mathbf{\eta}}$
whose stabiliser in $G$ is trivial.
Furthermore,
there exists a homogeneous $G$-semiinvariant $\Delta\in\k[\L]$ of degree $p^n-1$ with the
property that $P_{\mathbf{\eta}}^{\,\circ}=\,\{x\in P_{\mathbf{\eta}}\,|\,\,\Delta(x)\ne 0\}$.

The above discussion shows that almost all results of Kostant [\cite{K63}] on the adjoint action of a complex reductive group on its Lie algebra hold for the action of $G$ on $\L$. There is one notable exception though. Since it is proved in [\cite{P91}] that $P_{\bf 0}\setminus P_{\bf 0}^{\,\circ}$ coincides with the singular locus of $P_{\bf 0}$, the special fibre $P_{\bf 0}$ of $\psi$ is {\it not} a normal variety. Similar to the classical case that fibre coincides with $\NN(\L)=\{x\in\L\,|\,\, x^{p^n}=0\}$, the nilpotent cone of the restricted Lie algebra $\L$.

One of the main goals of this paper is to complete the picture by showing that a fibre $P_{\mathbf{\eta}}$ is a normal variety if and only if it is smooth and we demonstrate that the latter happens if and only if $\psi_0(x)\ne 0$ for all
$x\in P_{\mathbf{\eta}}$
This is established in Subsection~4.2 with the help of
some results obtained by Skryabin in [\cite{Sk14}]. Our arguments in Section~4 also rely on the description of regular elements of $\L$ obtained in Section~3.
\subsection{}
An element $x$ of a finite dimensional Lie algebra $\g$ is called {\it regular} if the centraliser $\mathfrak{c}_\g(x)$ has the smallest possible dimension. The set $\g_{\rm reg}$ of all regular elements of $\g$ is Zariski open in $\g$. In the classical situation of a reductive Lie algebra $\g$ over $\mathbb{C}$, Kostant discovered a criterion for regularity of $x\in\g$
based on the behaviour of the differentials ${\rm d}f_1,\ldots {\rm d}f_\ell$ at $x$ of a system of basic invariants $f_1,\ldots, f_\ell\in\mathbb{C}[\g]^\g$; see [\cite{K63}]. Various versions of Kostant's differential criterion for regularity have been recorded in the literature and it is clear that the criterion represents a repeating pattern in the invariant theory of group schemes; see [\cite{Sk02}, \cite{PPY07}], for example.

In Section~3, we prove that $D\in\L_{\rm reg}$ if and only if the differentials ${\rm d}\psi_0,\ldots {\rm d}\psi_{n-1}$ are linearly independent at $D$ and show that this happens if and only if the kernel of $D$ in $\OO_n$ is spanned by the identity element; see Theorems~\ref{t1} and \ref{t2}(ii). This result is then used to give a comprehensive description of all regular conjugacy classes in $\L$; see Theorem~\ref{t2}(iii).
\subsection{}
Finally, in Section~4 we show that there exists an irreducible $G$-semiinvariant $\Delta_0\in\k[\L]$ such that $\Delta=(-1)^n\Delta_0^{p-1}$ and we identify the restriction of $\Delta_0$ to $\t_0$ with the classical Dickson semiinvariant for $\mathrm{GL}_n(\mathbb{F}_p)$. As a consequence,
we obtain rather explicit formulae for the basic invariants $\psi_0,\ldots,\psi_{n-1}\in\k[\L]^G$ in the spirit of [\cite{Di11}]; see Subsection~4.1.

\medskip

\noindent{\bf Acknowledgements.} I would like to thank Hao Chang whose questions initiated this research. I am also thankful to J\"org Feldvoss for his interest and encouragement. Special thanks go to Serge Skryabin for pointing out a serious error in the first version of this paper.

\section{\bf Generalities and recollections}\label{sec1}
\subsection{}\label{ss1} Let $\k$ be an algebraically closed field of characteristic $p>2$ and write $\OO_n$ for the truncated polynomial ring $\k[X_1,\ldots,X_n]/(X_1^p,\ldots,X_n^p)$ in $n$ variables. Let $x_i$ denote the image of $X_i$ in $\OO_n$ and let $\m$ stand for the unique maximal ideal of $\OO_n$ (this ideal is generated by  $x_1,\ldots, x_n$). For every $f\in\OO_n$ there exists a unique
element $f(0)\in \k$ such that $f-f(0)\in \m$ and it is easy to see that $f^p=f(0)^p$ for all $f\in\OO_n$. Given an $n$-tuple
$(f_1,\ldots,f_n)\in \m^n$ we write ${\rm Jac}(f_1,\ldots,f_n)$ for the determinant the Jacobian matrix $\big(\frac{\partial f_i}{\partial x_j}\big)_{1\le i,j\le n}$ with entries in $\OO_n$.

Let $G$ be the automorphism group of the $\k$-algebra $\OO_n$. Each $\sigma\in
G$ is uniquely determined by its effect on the generators $x_i$ of $\OO_n$ and since $x_i^p=0$ it must be that $\sigma(x_i)\in\m$ for all $i$. An assignment $\sigma(x_i)=f_i$ with $f_i\in\m$ extends to an automorphism of $\OO_n$ if and only in ${\rm Jac}(f_1,\ldots, f_n)\not\in \m$. This shows that $G$ is a connected algebraic
$\k$-group whose unipotent radical $R_u(G)$ consists of those
$\sigma\in G$ for which $\sigma(x_i)-x_i\in\m$ for all $i$. Furthermore, $G/R_u(G)\cong {\rm GL}_n(\k)$ and there is a reductive subgroup $G_0$
isomorphic to ${\rm GL}_n(\k)$ such that $G\cong G_0
\ltimes R_u(G)$ as algebraic $\k$-groups. More precisely, $G_0$ consists of all automorphism of $\OO_n$ induced by the {\it linear} substitutions of the $x_i$'s.

Let $\L={\rm Der}(\OO_n)$, the Lie algebra of all derivations of $\OO_n$. Any $D\in \L$ is uniquely determined by its effect on the generators $x_1,\ldots, x_n$. Conversely, for every $n$-tuple $(f_1,\ldots,f_n)\in(\OO_n)^n$
there exists a unique $D\in\L$ such that $D(x_i)=f_i$ for all $i$.
We denote by $\partial_i$ the derivation of $\OO_n$ with the property that $\partial_i(x_j)=\delta_{ij}$ for $1\le j\le n$. The above discussion implies that $\L$ 
is a free $\OO_n$-module with basis consisting of $\partial_1,\ldots, \partial_n$.

Recall that $p>3$ and $(p,n)\ne (3,1)$. In this situation, Jacobson proved in [\cite{Jac43}] that any automorphism of $\L$ is induced by a unique automorphism of $\OO_n$ (this is stated under the assumption that $p>3$ in {\it loc. cit.}, but
after a slight modification Jacobson's arguments go through in our present case). So from now on we shall identify $G$ with the automorphism group ${\rm Aut}(\L)$ by using the rule $$\sigma(D)=\sigma\circ D\circ \sigma^{-1}\ \qquad(\forall\,\sigma\in G,\,D\in\L).$$

There is a unique cocharacter $\lambda\colon\,
\k^\times\to G$ such that $(\lambda(t))(x_i)=tx_i$
for all $t\in\k^\times$ and
$1\le i\le n$. Since $\sigma(fD)=\sigma(f)\sigma(D)$
for all $\sigma\in G$, $f\in \OO_n$ and $D\in \L$
and since $(\lambda(t))(\partial_i)=t^{-1}\partial_i$ for all $i$, the action of $\lambda(\k^\times)$ gives $\L$ a $\Z$-grading
$$\L=\L_{-1}\oplus\L_0\oplus\cdots \oplus\L_{n(p-1)-1},\qquad [\L_i,\L_j]\subseteq \L_{i+j},$$ such that
$\L_{-1}=\k\partial_1\oplus\cdots\oplus\k\partial_n$
and $\L_0\cong\mathfrak{gl}_n(\k)$. The subalgebra
$$\L_{(0)}:=\textstyle{\bigoplus}_{i\ge 0}\,\L_i=\m\partial_1\oplus\cdots\oplus\m\partial_n$$
is often referred to as the {\it standard maximal subalgebra} of $\L$. By a result of Kreknin, it can be characterised as the unique proper subalgebra of smallest codimension in $\L$. It is immediate from the above-mentioned description of $G$ that $\Lie(G)=L_{(0)}.$

For any $k\ge \Z_{\ge -1}$ we set $\L_{(k)}=\bigoplus_{i\ge k}\,\L_i$. It is straightforward to see that $$\L_{(1)}=\textstyle{\bigoplus}_{i\ge 0}\,\L_i=\m^2\partial_1\oplus\cdots\oplus\m^2\partial_n$$ is the nilradical of $\L_{(0)}$ and $\L_{(0)}/L_{(1)}\cong\mathfrak{gl}_n(\k)$.
\subsection{}\label{ss2} Let $\g$ be a finite dimensional restricted Lie algebra over $\k$ with  $p$-mapping $\pi\colon \g\to\g,\ x\to x^{[p]}$.
By Jacobson's formula, $\pi$ is a morphism of algebraic varieties induced by a collection of homogeneous polynomial functions of degree $p$ on $\g$. Given $k\in\mathbb{N}$ we write $\pi^k$ for the $k$-th iteration of $\pi$, so that $\pi^k(x)=x^{[p]^k}$ for all $x\in \g$.
An element $x$ of $\g$ is called {\it nilpotent}
if $\pi^N(x)=0$ for $N\gg 0$. The set $\NN(\g)$ of all nilpotent elements of $\g$ is a Zariski closed, conical subset of $\g$.
Given and element $x\in\g$ we set $\langle x,[p]\rangle:=\, \sum_{i\ge 0}\k \pi^i(x)$ of $\g$. This is an abelian restricted subalgebra of $\g$.
We say that $x\in\g$ is {\it semisimple}
if $x\in \langle x^{[p]},[p]\rangle=\, \sum_{i\ge 1}\k \pi^i(x)$. It is well known (and easy to see) that for any $x\in
\g$ there exist a unique semisimple element $x_s$ and a unique nilpotent element $x_n$ in $\langle x, [p]\rangle$ such that $x=x_s+x_n$. For a restricted subalgebra $\h$ of $\g$ we denote by $\h_s$ the set of all semisimple elements of $\h$. As $\h_s=\pi^N(\h)$, where $N\gg 0$, the set $\h_s$ contains a nonempty open subset of its Zariski closure in $\h$.

Given $x\in \g$ we let $\g_x^0$ denote the set of all
$y\in \g$ for which $(\ad x)^{N}(y)=0$, where $N\gg 0$, and we set ${\rm rk}(\g):=\min_{x\in\g}\,\dim\g_x^0$. It is well known that $\g_x^0$ is a restricted  subalgebra of $\g$ containing the centraliser $\mathfrak{c}_{\g}(x)$. Furthermore,
if $\dim \g_x^0={\rm rk}(\g)$ then $\g_x^0$
is a Cartan subalgebra of minimal dimension in $\g$. In particular, $\g_x^0$ is a nilpotent Lie algebra.
We say that an element $x\in\g$ is {\it regular} if its centraliser $\mathfrak{c}_\g(x)$ has the smallest possible dimension. It follows from basic linear algebra that the set $\g_{\rm reg}$ of all regular elements of $\g$ is Zariski open in $\g$.

A Lie subalgebra $\t$ of $\g$ is called {\it toral} (or a {\it torus}) if all elements of $\t$ are semisimple, and an element $t\in \g$ is called {\it toral} if $t^{[p]}=t$. As $\k$ is algebraically closed,
any toral subalgebra $\t$ of $\g$ is abelian. Furthermore,
the set $\t^{\rm tor}$ of all toral elements of $\t$
is an ${\Bbb F}_p$-subspace of $\t$ containing a $\k$-basis of $\t$. In particular, ${\rm Card}(\t^{\rm tor})=p^l$ where $l=\dim \t$.
We denote by $MT(\g)$ the maximal dimension of toral subalgebras of $\g$ and write $\NN(\g)$ for the set of all nilpotent elements of $\g$.

Let $n=\dim \g$, $s=MT(\g)$, and let $e=e(\g)$ be the smallest nonnegative integer such that $\pi^e(V)\subseteq \g_s$ for some nonempty Zariski open subset of $\g$. By [\cite{P89}, Theorem~2], there exist nonzero homogeneous polynomial functions $\psi_0,
\ldots,\psi_{s-1}$ on $\g$ such that $\deg \psi_i=p^{s+e}-p^{i+e}$ and
\begin{equation}\label{p-pol}\pi^{s+e}(x)+\textstyle{\sum}_{i=0}^{s-1}\,\psi_i(x)\pi^{i+e}(x)=0\qquad\quad(\forall\,x\in\g).\end{equation}
Moreover, it is immediate from [\cite{P89}, Lemma~4(2)] that the $\psi_i$'s are invariant under the action of the automorphism group of the restricted Lie algebra $\g$ on the coordinate ring $\k[\g]$. By [\cite{P03}, Theorem~4.2], we also have that $\psi_i\circ \pi=\psi_i^p$ for all $i$. This implies that the nilpotent cone $\NN(\g)$ coincides with the zero locus
of the ideal of $\k[\g]$ generated by $\psi_0,\ldots, \psi_{s-1}$. Since $\NN(\g)$ intersects trivially with any $s$-dimensional torus of $\g$, all irreducible components of $\NN(\g)$ have dimension $n-s$. This implies that the polynomial functions $\psi_0,\ldots,\psi_{s-1}$ form a regular sequence in $\k[\g]$. In particular, they are algebraically independent in $\k(\g)$.

Since our base field has positive characteristic, it may happen that
$\g$ contains maximal tori of different dimensions. However, it is still true that if $\t$ is a maximal torus of $\g$ then its centraliser $\mathfrak{c}_\g(\t)$ is a Cartan subalgebra of $\g$. Conversely, any Cartan subalgebra $\h$ of $\g$ has the form $\h=\mathfrak{c}_\g(\t)$ where $\t=\h_{s}$
is the unique maximal torus of $\g$ contained in $\h$.
 A Cartan subalgebra $\h$ of $\g$ is called {\it regular} if $\dim\h_{s}=MT(\g)$. By the main results of
[\cite{P86}], a Cartan subalgebra $\h$ of $\g$ is regular if and only if the variety $\h\cap\NN(\g)$
has the smallest possible dimension. Furthermore,
the equalities $$\dim \h={\rm rk}(\g)\ \mbox{ and }\,\,   \dim \h-MT(\g)=\dim\g-\dim\overline{\pi_e(\g)}$$ hold for
any regular Cartan subalgebra of $\g$. In particular, $\g$ contains a self-centralising torus if and only if $e(\g)=0$, that is $\g_s$ contains a nonempty Zariski open subset of $\g$.  The above also shows that
any self-centralising torus $\t$ of $\g$ (if it exists) is a regular Cartan subalgebra of $\g$ and
$\dim \t=MT(\g)={\rm rk}(\g)$.

\section{\bf Characterising regular derivations of $\OO_n$}\label{sec2}
\subsection{}\label{3.1}
Given a torus $\t$ in a finite dimensional restricted Lie algebra $\g$ we denote by $(\t^{\rm tor})^*$ the set of all linear functions $\alpha\colon\,\t\to \k$ such that $\alpha(t)\in \mathbb{F}_p$ for all toral elements $t\in\t$.
Our discussion in Subsection~\ref{ss2} shows that
$(\t^{\rm tor})^*$ is an $\mathbb{F}_p$-form of the dual space $\t^*$. In particular, ${\rm Card}((\t^{\rm tor})^*)=p^l$ where $l=\dim \t$. Any finite dimensional restricted $\g$-module $V$ decomposes as $V=\bigoplus_{\lambda\in\t^*\,} V^\lambda$ where $$V^\lambda=\{v\in V\,|\,\,t.v=\lambda(t)v\, \mbox{ for all }\ t\in \t\}.$$ We say that $\lambda\in\t^*$ is a {\it weight} of $V$ with respect to $\t$ (or a $\t$-{\it weight}) if $V^\lambda\ne\{0\}$ and we write $\Lambda(V)$ for the set of all $\t$-weights of $V$. It is immediate from the definitions that $\Lambda(V)\subseteq (\t^{\rm tor})^*$.

For any $D\in\L$ the endomorphism $D^p\in\mathfrak{gl}(\OO_n)$ is a derivation of $\OO_n$. Therefore, $\L$ carries a natural restricted Lie algebra structure. Since $\z(\L)=\{0\}$, this structure is unique. In particular, it is equivariant under the action of $G$ on $\L$. We mention that $\OO_n$ is tautologically a restricted
$\L$-module and so the notation $\Lambda(\OO_n)$ makes sense for any toral subalgebra $\t$ of $\L$.
\begin{lemma}\label{l1}
Let $\t$ be an $r$-dimensional torus in $\L$ and let $\Lambda(\OO_n)$ be the set of all $\t$-weights of $\OO_n$. Then $\Lambda(\OO_n)=(\t^{\rm tor})^*$ and $\dim \OO_n^\lambda=p^{n-r}$ for all $\lambda\in\Lambda(\OO_n)$.
\end{lemma}
\begin{proof} Let $\t_{(0)}=\t\cap\L_{(0)}$ and let $t_1,\ldots, t_s$ be toral elements of $\t$ whose images in $\t/\t_{(0)}$ form a basis for that vector space. We first suppose that $s\ge 1$.
Replacing $\t$ by $\sigma(\t)$ for a suitable
$\sigma\in G$ (if required)  we may assume that
$t_i=\{(1+x_i)\partial_i$ for all $1\le i\le s$
and $\t_{(0)}$ is an
$(r-s)$-dimensional subtorus of
$\t_{n,s}:=\,\bigoplus_{i=s+1}^n\k x_i\partial_i$; see [\cite{St1}, Theorem~7.5.1].
Let $\t_{0,s}:=\bigoplus_{i=1}^s\,\k(1+x_i)\partial_i$, an $s$-dimensional torus of $\L$, and let $\OO_{n,s}$ be the subalgebra of $\OO_n$ generated by $x_{s+1},\ldots, x_n$. It is straightforward to see that $\OO_{n,s}\cong\OO_{n-s}$ as $\k$-algebras and $\OO_n$ is a free
$\OO_{n,s}$-module with basis
$$\mathcal{X}:=\{(1+x_{1})^{a_1}\cdots(1+x_s)^{a_s}\,|\,\,0\le a_i\le p-1\}.$$ The set $\mathcal{X}$
consists of weight vectors for $\t_{0,s}$ corresponding to pairwise distinct weights and $t(\mathcal{X})=0$ for all  $t\in\t_{n,s}$. Since $\t=\t_{0,s}\oplus\t_{(0)}$, this shows that in proving the lemma we may assume without loss of generality that $s=0$.

Now suppose that $\t\subseteq \L_{(0)}$. In this situation [\cite{St1}, Theorem~7.5.1] essentially says that $\t$ may be assumed to be an
$r$-dimensional subtorus of
$\t_n$.
The normaliser $N$ of the set $\{x_1,\ldots, x_n\}$ in $G$ is isomorphic to the symmetric group $\mathfrak{S}_n$ and permutes the set $\{x_1\partial_1,\ldots,x_n\partial_n\}$ which forms a basis of the vector space $(\t_n)^{\rm tor}$ over $\mathbb{F}_p$. Keeping this in mind one observes that there is $\sigma\in N$ for which $\sigma(\t^{\rm tor})=\mathbb{F}_pt_1\oplus\cdots\oplus\mathbb{F}_pt_r$ where
$$t_i=x_i\partial_i+\sum_{j=r+1}^nc_{i,j}(x_j\partial_j)\ \qquad\quad
(1\le i\le r)$$
for some $c_{i,j}\in\mathbb{F}_p$.
Let $A_n$ denote the set of all $n$-tuples ${\bf a}=(a_1\ldots, a_n)$ such that $0\le a_i\le p-1$ and set $x^{\bf a}:=x_1^{a_1}\cdots x_n^{a_n}$. The set $\{x^{\bf a}\,|\,\, {\bf a}\in A_n\}$ is a basis of $\OO_n$ consisting of weight vectors for $\t_n$ and we have that
$$t_i(x^{\bf a})=\big(a_i+\textstyle{\sum}_{j=1}^nc_{i,j}a_j\big)x^{\bf a}\qquad (1\le i\le r)$$
(in order to ease notation we identify the $a_i$'s with their images in $\mathbb{F}_p=\Z/p\Z$). In the spirit of a first year linear algebra course we solve the system of linear equations
\begin{eqnarray*}
a_1+\cdots+c_{1,r+1}a_{r+1}+\cdots+c_{1,n}a_n &=&m_1\\
\ddots\qquad\qquad\qquad\qquad\qquad\qquad\vdots\ & &\ \vdots\\
a_r+c_{r,r+1}a_{r+1}+\cdots+c_{r,n}a_n&=&m_r
\end{eqnarray*}
for any $r$-tuple $(m_1,\ldots, m_r)\in (\mathbb{F}_p)^r$
by declaring $a_{r+1},\ldots, a_n$ to be our ``free variables''. We then deduce that the number of solutions with coefficients in $\mathbb{F}_p$ equals $p^{n-r}$. This completes the proof.
\end{proof}
\begin{corollary}\label{c1}
Let $\t$ be an $r$-dimensional torus in $\L$ and let $\Lambda(\L)$ be the set of all $(\ad \t)$-weights of $\L$. Then $\Lambda(\L)=(\t^{\rm tor})^*$ and $\dim \L^\lambda=np^{n-r}$ for all $\lambda\in\Lambda(\OO_n)$.
\end{corollary}
\begin{proof}
Thanks to [\cite{St1}, Theorem~7.5.1] it can be assumed without loss that $\partial_1,\ldots, \partial_n$ are weight vectors for $\t$. Let  $\mu_i\in \Lambda(\L)$ be the weight of $\partial_i$, where $1\le i\le n$, and take any $\lambda\in(\t^{\rm tor})^*$. Since $$[D,fD']=D(f)D'+f[D,D']\qquad(\forall\,D,D'\in \L,\,f\in \OO_n).$$
we have that
$\L^\lambda=\textstyle{\sum}_{i=1}^n\,\OO_n^{\lambda-\mu_i}\partial_i.$
Since $\lambda-\mu_i\in (\t^{\rm tor})^*$, it follows from Lemma~\ref{l1} that $\dim\OO^{\lambda-\mu_i}=p^{n-r}$ for all $i$.
As a result, $\dim\L^\lambda=np^{n-r}$ for all $\lambda\in (\t^{\rm tor})^*$ as stated.
\end{proof}
\subsection{}\label{3.2} According to [\cite{P91}, Theorem~1], the characteristic
polynomial of any derivation $x\in\L$ has the form
$$\det(tI_{\OO_n}-x)=t^{p^n}+\textstyle{\sum}_{i=0}^{n-1}\psi_{i}(x)t^{p^i},$$ and the homogeneous polynomial functions
$\psi_0,\ldots,\psi_{n-1}$ generate freely the invariant algebra $\k[\L]^G$. By the Cayley--Hamilton theorem, we have that
\begin{equation}\label{ppol}
x^{[p^n]}+
\textstyle{\sum}_{i=0}^{n-1}\,\psi_{i}(x)x^{[p^i]}=0\qquad\quad(\forall\,x\in\L).
\end{equation}
Due to Lemma~\ref{l1} and Corollary~\ref{c1} this shows that
$x\in \L$ is regular semisimple if and only if $\psi_0(x)\ne 0$. From this it follows that $e(\L)=0$.
Our discussion in Subsection~\ref{ss2} now yields that
$\psi_i(x^{[p^k]})=\psi_i(x)^{p^k}$ for all $k\in\Z_{\ge 0}$. As $x^{[p^k]}=x_s^{[p^k]}$ for all $k\gg 0$, this entails that $\psi_i(x)=\psi_i(x_s)$
for all $x\in\L$.

Given an arbitrary element $D\in \L$ we denote by $\t_D$ the torus of $\L$ generated by the semisimple part $D_s$ of $D$.
\begin{lemma}\label{l2}
Let $D\in\L\setminus\mathcal{N}(\L)$ and define $r=r(D):=\min\{0\le i\le n-1\,|\,\,\psi_i(D)\ne 0\}$.
Then the following hold:
\begin{itemize}
\item[(i)\,] $\dim \t_D=n-r$ and $D_n^{p^r}=0$.

\smallskip

\item[(ii)\,] The linear map $Q(D):=(\ad D)^{p^n-p^r}+\sum_{i=r}^{n-1}\psi_i(D)(\ad D)^{p^i-p^r}\in\End(\L)$
acts invertibly on $\L_D^0$ and annihilates each $\t_D$-weight space $\L^\lambda$ with $\lambda\ne 0$.
\end{itemize}
\end{lemma}
\begin{proof}
By (\ref{ppol}), we have that $D^{p^n}+
\sum_{i=r}^{n-1}\,\psi_{i}(D)D^{p^i}=0$. Since $\psi_r(D)\ne 0$ this shows that $D^{p^r}$ lies in the
restricted subalgebra of $\L$ generated by $D^{p^i}
$ with $i>r$. Therefore, $D^{p^r}$ is a semisimple element of $\L$. As $D^{p^r}=D_s^{p^r}+D_n^{p^r}$ and $D_n^{p^r}\in \L$ is nilpotent, it must be that $D_n^{p^r}=0$.
As $\psi_i(D)=\psi_i(D_s)$ by our earlier remarks, we have that $$D_s^{p^n}+
\textstyle{\sum}_{i=r}^{n-1}\,\psi_{i}(D)D_s^{p^i}=0.$$
As $\psi_r(D)\ne 0$, this shows that the kernel of $D_s$ has dimension $p^r$. But $\ker D_s$ is nothing but the
zero weight space of $\t_D$. So  Lemma~\ref{l1} yields $\dim \t_D=n-r$ and statement~(i) follows.

Since $\t_D$ is generated by $D_s$ as a restricted Lie algebra, $\ad D_s$ must act invertibly on any weight space $\L^\lambda$ with
$\lambda\ne 0$. Since $(\ad D)^{p^r}=(\ad D_s)^{p^r}$ by part~(i), it is immediate from (\ref{ppol}) that $Q(D)$ must annihilate all such $\L^\lambda$. Finally, $Q(D)$ acts invertibly on $\L_D^0$ because
$\psi_r(D)\ne 0$ and the restriction of $\ad D$ to $\L_D^0$ coincides with that of $\ad D_n$.
\end{proof}
\subsection{}\label{3.3}
We are now ready to prove a differential criterion for regularity of elements in $\L$. It is a precise analogue of Kostant's classical result [\cite{K63}, Theorem~0.1]. Recall that the morphism $\psi\colon\,\L\to\mathbb{A}^n$ sends any
$x\in \L$ to $\big(\psi_0(x),\ldots,\psi_{n-1}(x)\big)\in\mathbb{A}^n$ and we denote by $P_{\mathbf {\eta}}$ the inverse image of ${\mathbf{\eta}}\in\mathbb{A}^n$ under $\psi$.
\begin{theorem}\label{t1} Let $D\in \L$ and
suppose that the torus $\t_D$ generated by $D_s$ has dimension $n-r$ where $0\le r\le n$. Then
the following are equivalent:
\begin{itemize}
\item[(i)\,] $D$ is a regular element of $\L$.

\smallskip

\item[(ii)\,] The differentials
${\rm d}\psi_0,\ldots,{\rm d}\psi_{n-1}$ are linearly independent at $D$.

\smallskip

\item[(iii)\,] $D$ is a smooth point of the fibre $P_{\psi(D)}$.
\end{itemize}
\end{theorem}
\begin{proof}
Thanks to Jacobson's formula for $p$th powers, replacing $x$ by $D+ty$ in equation~(\ref{ppol}) and computing the coefficient of $t$ we obtain that
\begin{equation}\label{dpsi}
(\ad D)^{p^n-1}(y)+\textstyle{\sum}_{i=r}^{n-1}\,\psi_i(D)(\ad D)^{p^i-1}(y)=-\textstyle{\sum}_{i=0}^{n-1}\,({\rm d}\psi_i)_D(y)\cdot D^{p^i}\quad\ (\forall\,y\in\L).
\end{equation}
Using the notation introduced in Subsection~\ref{3.2}
we can rewrite (\ref{dpsi}) as follows:
\begin{equation}\label{dpsi1}
\big((\ad D)^{p^r-1}\circ Q(D)\big)(y)=-\textstyle{\sum}_{i=0}^{n-1}\,({\rm d}\psi_i)_D(y)\cdot D^{p^i}\qquad(\forall\,y\in\L).
\end{equation}
In conjunction with Lemma~\ref{l2}(ii) this shows that
$(\ad D)^{p^r-1}$ maps $\L_D^0$ into the linear span of $D,\ldots, D^{p^{n-1}}$. We stress that if $y\in\L^\lambda$ and $\lambda\ne 0$ then both sides of (\ref{dpsi}) must vanish as the LHS lies in $\L^\lambda$ whilst the RHS lies in
$\L^0$.

Suppose $D\in \L_{\rm reg}$. Since $\dim\L_D^0=np^r$ by Corollary~\ref{c1}, $D_n^{p^r}=0$ by Lemma~\ref{l2}(i), and $\mathfrak{c}_{\L}(D)\subseteq \L_D^0$ has dimension $n$, all Jordan blocks of the restriction $\ad D$ to $\L_D^0$ must have size $p^r$. In view of (\ref{dpsi1}), this implies that $D,\ldots, D^{p^{n-1}}$ are
linearly independent in $\L$. Thanks to (\ref{dpsi}) this yields that all linear functions
$({\rm d}\psi_i)_D$ vanish on the subspace $[D,\L]$ and the map
$\delta\colon\,\L/[D,\L]\to {\rm span}\{D,\ldots, D^{p^{n-1}}\}$ given by
$$\delta\big(y+[D,\L]\big)=-\textstyle{\sum}_{i=0}^{n-1}\,({\rm d}\psi_i)_D(y)\cdot D^{p^i}\qquad\quad (\forall\,y\in \L)$$ is a linear isomorphism.
From this it is immediate that the differentials
${\rm d}\psi_0,\ldots, {\rm d}\psi_{n-1}$ are linearly independent at $D$.

Now suppose that (ii) holds for $D$. If $r=n$ then $D$ is nilpotent. In this case [\cite{P91}, Theorem~2] shows that $D\in\L_{\rm reg}$. So let us assume from now that $r\le n-1$. Then, of course, $D\ne 0$. On the other hand, our present assumption on $D$ implies that the linear map
$X\mapsto\big(({\rm d}\psi_0)(X),\ldots, ({\rm d}\psi_{n-1})(X)\big)$
from $\L$ to $\k^n$  is surjective. Combining (\ref{dpsi1}) with Lemma~\ref{l2}(ii) yields that there exists $y\in\L_D^0$ such that $$D=(\ad D)^{p^r-1}(y)=(\ad D_n)^{p^r-1}(y).$$  As a consequence,  $S_1:=\{D_n^{p^i}\,|\,\,0\le i\le r-1\}$ is a linearly independent set. On the other hand,  $S_2:=\{D^{p^i}\,|\,\,r\le i\le n-1\}$ consists of semisimple elements of $\L$ and is linearly independent because $\dim\t_D=n-r$. As $S_1\cup S_2\subseteq{\rm span}\{D^{p^i}\,|\,\,0\le i\le n-1\}$ and
${\rm span}(S_1)\cap{\rm span}(S_2)=\{0\}$, we now deduce that the set $\{D^{p^i}\,|\,\,0\le i\le n-1\}$
is linearly independent. In view of (\ref{dpsi1}) this enables us to conclude that all Jordan blocks of the restriction of $\ad D$ to $\L_D^0$ have size $p^r$. Since $\ker \ad  D\subseteq \L_D^0$, applying Corollary~\ref{c1} yields $D\in\L_{\rm reg}$.

We have proved that (i) and (ii) are equivalent. According to [\cite{P91}, Lemma~13], the fibre $P_{\psi(D)}$ is irreducible and its defining ideal in $\k[\L]$ is generated by the polynomial functions $\psi_0-\psi_0(D),\ldots, \psi_{n-1}-\psi_{n-1}(D)$.
From this it is immediate that (ii) is equivalent to (iii).
\end{proof}
\begin{rem}\label{r1}
It follows from the proof of Theorem~\ref{t1} that
if $D\in\L_{\rm reg}$ and $r=r(D)$, then $(\ad D)^{p^r-1}$ maps $\L_D^0$ onto $\mathfrak{c}_{\L}(D)=\langle D,[p]\rangle=\sum_{i=0}^{n-1}\k D^{p^i}$ and the derivations $D,D^p,\ldots, D^{p^{n-1}}$ are linearly independent.
\end{rem}
\subsection{}\label{3.4} Our next goal is to give a more explicit characterisation of the elements in $\L_{\rm reg}$.
\begin{theorem}\label{t2}
Suppose $D\in\L$ and let $r=r(D)$. Then the following are equivalent:
\begin{itemize}
\item[(i)\,] $D$ is a regular element of $\L$.

\smallskip

\item[(ii)\,] $\ker D=\k 1$ is $1$-dimensional.

\item[(iii)\,] There exist $z_{r+1},\ldots, z_n\in\{\epsilon_i+x_i\,|\,\,r+1\le i\le n\}$ for some $\epsilon_i\in\{0,1\}$ and $\sigma\in G$ such that $\sigma(D_s)=
\sum_{i=1}^{n-r}\lambda_i (z_i\partial_i)$ for some $\lambda_i\in\k$, the torus  $\sigma(\t_D)$ is spanned by $z_{r+1}\partial_{r+1},\ldots, z_n\partial_n$, and $\sigma(D_n)=\partial_1+
x_1^{p-1}\partial_2+\cdots+x_1^{p-1}\cdots x_{r-1}^{p-1}\partial_r$.

\smallskip

\item[(iv)\,] All Jordan blocks of $D_n$ have size $p^r$.
\end{itemize}
\end{theorem}
\begin{proof}
(a) Suppose $D\in\L_{\rm reg}$. We wish to prove that
$\ker D=\k 1$. So suppose the contrary. Then $\ker D$
has dimension $\ge 2$ ad hence the subspace $ \m\cap\ker D$ contains a nonzero element, $f$ say.
It is straightforward to see that if $f\in\m\setminus \m^2$ then $f^2\ne 0$ (here we use our assumption on $p$). Therefore, no generality will be lost by assuming that $f\in\m^2$. Since $fD\in\mathfrak{c}_{\L}(D)$
and $(fD)^p=f^pD^p=0$, it follows from  Remark~\ref{r1} that
$fD=\lambda D_n^{p^{r-1}}$ for some $\lambda\in\k$.

Suppose $D\not\in\L_{(0)}$. The $fD\ne 0$ and after rescaling $f$ (if need be) we may assume that
$fD=D_n^{p^{r-1}}$. This yields $D_n^{p^{r-1}}\in \L_{(1)}$.
As the restriction of $\ad D$ to $\L_D^0$ coincides with that of $\ad D_n$, it follows from Remark~\ref{r1} that $(\ad D_n)^{p^r-1}(y)=D$ for some $y\in\L_D^0$. Since $p-1\ge 2$, this yields
\begin{equation}\label{eqq}
D=\big(\ad D_n^{p^{r-1}}\big)^{p-1}\big((\ad D_n)^{p^{r-1}-1}(y)\big)\in [\L_{(1)},[\L_{(1)},\L]]\subseteq \L_{(1)}.
\end{equation}
This contradiction shows that the present case cannot occur.

Suppose $D\in\L_{(0)}$. Quite surprisingly, this case is more complicated, but the good news is that due to [\cite{St1}, Theorem~7.5.1] we now may assume that $\t_D\subseteq \t_n$.
Recall that $\{x_1\partial_1,\ldots,x_n\partial_n\}$
is a basis of $\t_n$ contained in $\t_n^{\rm tor}$.
Let
$\{\varepsilon_1,\ldots, \varepsilon_n\}$ be the corresponding dual basis in $\t_n^*$, so that $\varepsilon_i\big((x_j\partial_j)\big)=\delta_{ij}$ for $1\le i,j\le n$, and denote by $\nu_i$ the restriction of $\varepsilon_i$ to $\t_D$.
Let $\ad_{\!-1}$
denote the representation of $\L_0$ in $\mathfrak{gl}(\L_{-1})$ induced by the adjoint action of $\L_0$ to $\L_{-1}$. Then $\Lambda:=\{-\nu_1,\ldots, -\nu_n\}$ coincides with
the set of weights of $\t_D$ on $\L_{-1}$. Since $\dim\t_D=n-r$ and $\Lambda$ spans the dual space $\t_D^*$, it must be that ${\rm Card}(\Lambda)\ge n-r$.

For $\nu\in\Lambda$ we set $m(\nu):=\dim \L_{-1}^{\nu}$. Then $m(\nu)\ge 1$ and $\sum_{\nu\in\Lambda}\,m(\nu)=n$. Put $m:=\max_{\nu\in\Lambda}\,m(\nu)$ and pick $\nu'\in\Lambda$ such that $m(\nu')=m$.
Then
\begin{equation}\label{eqn}
m=n-\textstyle{\sum}_{\nu\in\Lambda\setminus
\{\nu'\}}\,m(\nu)\le n-{\rm Card}\big(\Lambda\setminus
\{\nu'\}\big)\le n-(n-r-1)=r+1.
\end{equation}
Write $D_n=\sum_{i\ge 0}\,D_{n,i}$ where $D_{n,i}\in \L_i$. Since $\t_D\subseteq \t_n$ we have that
$D_{n,i}\in\mathfrak{c}_\L(\t_D)$ for all $i$.
In particular, this means that $\ad_{\!-1}(D_{n,0})$
preserves each weight space $\L_{-1}^\nu$. Since $D_n$ is nilpotent, this yields $\big(\ad_{\!-1}(D_{n,0})\big)^m=0$.
As $\L_{(1)}$ is a restricted ideal of $\L_{(0)}$, it follows from Jacobson's formula that $D_n^{p^k}-D_{n,0}^{p^k}\in \L_{(1)}$ for all $k\in\Z_{\ge 0}$.

Suppose $r\ge 2$. Since $p\ge 3$, easy induction or $r$ shows that $p^{r-1}\ge r+1$. In view of (\ref{eqn}) this implies that $\ad_{\!-1}\big(D_{n,0}^{p^{r-1}}\big)=0$. But  $
D_{n,0}^{p^{r-1}}=0$ because $\ad_{\!-1}$ is a faithful representation of $\L_0$. Due to our earlier remarks, this forces $D_n^{p^{r-1}}\in \L_{(1)}$. So we can again apply (\ref{eqq}) to conclude that $D\in\L_{(1)}$. In particular, $D$ is nilpotent. As $D\in\L_{\rm reg}$, this contradicts [\cite{P91}, Theorem~2].

Now suppose $r=1$. Then ${\rm Card}(\Lambda)\in\{n-1,n\}$.
If ${\rm Card}(\Lambda)=n$ then the equality $\sum_{\nu\in \Lambda} m(\nu)=n$ forces $m=1$. So $\ad_{\!-1}(D_{n,0})=0$ which again yields $D_n^{p^{r-1}}=D_n\in\L_{(1)}$. So we can reach a contradiction by arguing as before.

If ${\rm Card}(\Lambda)=n-1$ then the above reasoning shows that $m=2$. Since $\dim \t_D=n-1$, we may assume without loss that $\t_D$ is spanned by elements
$$t_i=x_i\partial_i+c_i(x_n\partial_n)\ \qquad \quad(1\le i\le
n-1)$$ for some $c_i\in\mathbb{F}_p$ (see the proof of Lemma~\ref{l1} for more detail). In this case $\nu_1,\ldots, \nu_{n-1}$ are linearly independent and $\nu_n=c_1\nu_1+\cdots+c_{n-1}\nu_{n-1}$. As ${\rm Card}(\Lambda)=n-1$, it must be that $\nu_n=\nu_k$ for some $k\le n-1$, so that
$c_k=1$ and $c_i=0$ for $i\ne k$. Since $D_{n,0}\in\mathfrak{c}_{\L_0}(\t_D)$, replacing $D$
by $\sigma(D)$ for a suitable $\sigma\in G_0\cong {\rm GL}_n(\k)$ we may assume further that $k=n-1$ and $D_{n,0}=\lambda(x_{n-1}\partial_n)$ where
$\lambda\in\k$. Also, $$D_s=\textstyle{\sum}_{i=1}^{n-2}\,\alpha_i(x_i\partial_i)+\alpha_{n-1}(x_{n-1}\partial_{n-1}+x_n\partial_n)$$
for some $\alpha_i\in\k^\times$. As a consequence, $\mathfrak{c}_{\L}(\t_D)\subset \L_{(0)}$.

If $\lambda=0$ then $D_n\in\L_{(1)}$ and we can apply (\ref{eqq}) with $r=1$ to conclude that $D\in\L_{(1)}$. As $D\in\L_{\rm reg}$, this contradicts [\cite{P91}, Theorem~2].
If $\lambda\ne 0$ then
$D_n\not\in\L_{(1)}$.
Note that $\L_D^0\cap \L_0$ is spanned by $x_i\partial_i$ with $1\le i\le n-2$ and $x_i\partial_j$ with $i,j\in\{n-1,n\}$. Since $\dim \L_D^0 = np$ by Corollary~\ref{c1} and $\L_D^0=\mathfrak{c}_\L(\t_D)\subset \L_{(0)}$ by our earlier remarks, $\L_D^0\cap\L_{(1)}$ is a nonzero ideal of $\L_D^0$.
Since $D_n$ is a nilpotent element of $\L_D^0$, we then  have
$\mathfrak{c}_{\L}(D)\cap \L_{(1)}\ne \{0\}$. On the other hand,
Remark~\ref{r1} shows that any nilpotent element of $\mathfrak{c}_{\L}(D)$ is a scalar multiple of
$D_n$ (it is important here that $r=1$). Since $D_n\not\in\L_{(1)}$, we reach a contradiction thereby showing that the present case cannot occur.

Finally, suppose $r=0$. Then $D_n=0$ and $\t_D=\t_n$ which entails that $\ker D$ coincides with $\k 1$, the zero weight space of $\t_n$ in $\OO_n$. We thus conclude that (i) implies (ii).

\smallskip

\noindent (b) Next suppose $\ker D=\k 1$ and let
$B=\{f\in \OO_n\,|\,\,t(f)=0\mbox{ for all }\, t\in\t_D\}$, the zero weight space of $\t_D$ in $\OO_n$. Let $\m_B=B\cap\m$, the maximal ideal of the local ring $B$. The restriction of $D$ to $B$ is a nilpotent derivation of $B$. Since the ideal $\m_B$ is not $D$-stable by our assumption on $D$, the algebra $B$ is differentiably simple. Since $\dim B=p^r$ by Lemma~\ref{l1}, Block's theorem yields that
$B\cong\OO_r$ as $\k$-algebras; see [\cite{Bl69}, Theorem~4.1]. Since $D_{\vert B}\in\Der(B)$ is nilpotent and $\ker D_{\vert B}=\k 1$, it follows from [\cite{P91}, Theorem~2] that there exist $y_1,\ldots,
y_r\in \m_B$ whose cosets in $\m_B/\m_B^2$ are linearly independent such that
$$D_{\vert B}=\frac{\partial}{\partial y_1}+y_1^{p-1}\frac{\partial}{\partial y_2}+\cdots+
y_1^{p-1}\cdots y_{r-1}^{p-1}\frac{\partial}{\partial y_r}.$$
We claim that the partial derivatives $\partial_i/\partial y_i\in\Der(B)$  can be extended to derivations of $\OO_n$. Indeed, [\cite{P91}, Lemma~3] implies that $\Der(B)$ is a free $B$-module with basis $D_{\vert B},\ldots, D_{\vert B}^{p^{r-1}}$. As consequence, there exists a subset $\{b_{ij}\,|\,\,0\le i,j\le r-1\}\subset B$ such that $$\partial_i/\partial y_i=\big(\textstyle{\sum}_{j=0}^{r-1}b_{ij}D_n^{p^j}\big)_{\vert B}\,\qquad\quad(1\le i\le r).$$ The claim follows.
Since each derivation $\partial_i/\partial y_i$ of $\OO_n$ maps $B\cap \m^2$ to $\m_B=\m\cap B$, we now deduce that $\m_B^2= B\cap\m^2$.
As a consequence, $\m_B/\m_B^2$ embeds into $\m/\m^2$.

It is immediate from the above discussion that
there are $y'_{r+1},\ldots, y'_n\in \m$ such that the cosets of $y_1,\ldots, y_r, y'_{r+1},\ldots, y'_n$ in $\m/\m^2$ are linearly independent. Since $\t_D$ acts semisimply on $\OO_n/\k 1$ and $\dim \t_D=n-r$, we may assume further that there exist $\gamma_{r+1},\ldots, \gamma_n\in (\t_D^{\rm tor})^*$ such that $y'_i+\k 1\in (\OO_n/\k1)^{\gamma_i}$. Since $\dim \t_D=n-r$ and $B=\OO_n^0$, the weights $\gamma_{r+1},\ldots,\gamma_n$ must form a basis of the dual space $\t_D^*$. 

By construction, $t(y'_i)=\gamma_i(t)y'_i+\gamma'_i(t)1$ for all $t\in\t_D$, where $\gamma_i'$ is a linear function on $\t_D$. If $\gamma_i'$ is not proportional to $\gamma_i$ then there exists $t_i\in\t_D$ with $\gamma_i(t_i)=0$ and $\gamma_0'(t_i)=1$. But then $t_i(y_i')=1$ and $t_i^p(y_i')=0$ contradicting the inclusion $t_i\in\langle t_i,[p]\rangle$.
Hence for each $i\ge r+1$ there is $\epsilon_i\in\k$ such that  $t(y_i'+\epsilon_i)=\gamma_i(t)(y_i'+\epsilon_i)$.
Rescaling the $y'_i$'s if need be we may assume that $\epsilon_i\in\{0,1\}$.
For $r+1\le i\le n$ we now set $y_i:=y'_i+\epsilon_i$.

By our choice of $y_1,\ldots, y_n$ there is a unique automorphism $\sigma$ of $\OO_n$ such that
$\sigma^{-1}(x_i)=y_i$ for $1\le i\le r$ and $\sigma^{-1}(x_i)=y_i-\epsilon_i$ for $r+1\le i\le n$. In view of our earlier remarks we  have that
$\sigma(D_n)=\partial_1+
x_1^{p-1}\partial_2+\cdots+x_1^{p-1}\cdots x_{r-1}^{p-1}\partial_r$ and
$\sigma(D_s)=
\sum_{i=r+1}^{n}\lambda_i (\epsilon_i+x_i)\partial_i$ for some $\lambda_i\in\k$. Since $\dim \t_D=n-r$ and all elements $(\epsilon_i+x_i)\partial_i$ are toral,
it is straightforward to see that $\sigma(\t_D)$ is spanned by $(\epsilon_{r+1}+x_{r+1})\partial_{r+1},\ldots,
(\epsilon_n+x_n)\partial_n$. This shows that
(ii) implies (iii).

\smallskip

\noindent
(c) Suppose (iii) holds for $D$ and adopt the notation introduced in part~(b). As before, we identify the elements  $a_i\in \{0,1,\ldots, p-1\}$ with their images in $\mathbb{F}_p$. Let $\gamma=\sum_{i=r+1}^n\, a_i\gamma_i$ be an arbitrary element of $(\t_D^{\rm tor})^*$, so that $a_i\in\mathbb{F}_p$. It  follows from Lemma~\ref{l1} that
the weight space $\OO_n^\gamma$ is a free $B$-module of rank $1$ generated by $y^\gamma:=y_{r+1}^{a_{r+1}}\cdots y_n^{a_n}$. It follows from [\cite{P91}, Lemma~7(iv)] that $(D_{\vert B})^{p^r-1}\ne 0$. Since $D_n(y^\gamma)=0$, we now see that $D_n$ acts on each $\OO_n^\gamma$ as a Jordan block of size $p^r$. This shows that (iii) implies (iv).

Finally, if (iv) holds for $D$ then Lemma~\ref{l1} yields that $D_n$ acts on the zero weight space $\OO_n^0$ for $\t_D$ as a single Jordan block of size $p^r$. Since $\ker D\subseteq \OO_n^0$, this forces $\ker D=\k 1$. So (iv) implies (ii). Since we have already established that (i) and (ii) are equivalent, our proof of Theorem~\ref{t2} is complete.
\end{proof}
\begin{rem}\label{r2}
The proof of Theorem~\ref{t2} also shows that $D\in\L_{\rm reg}$ if and only if $\OO_n^\lambda$ is a free $\OO_n^0$-module of rank $1$ and $\dim (\OO_n^\lambda\cap \ker D_n)=1$ for any $\t_D$-weight $\lambda\in\Lambda(\OO_n)$. On the other hand, we know from linear algebra that $x\in\mathfrak{gl}(\OO_n)_{\rm reg}$ if and only if the minimal polynomial of $x$ coincides with the characteristic polynomial of $x$.
From this it follows that
$$\L_{\rm reg}\,=\,\L\cap\mathfrak{gl}(\OO_n)_{\rm reg}.$$
\end{rem}
\section{\bf Dickson invariants and the fibres of $\psi$}
\subsection{}\label{4.1}
It follows from [\cite{P91}, Theorem~3] that the morphism $\psi$ which sends any $x\in\L$ to $\big( \psi_0(x),\ldots,
\psi_{n-1}(x)\big)\in\mathbb{A}^n$ is flat and surjective. Moreover, for any $\mathbf{\eta}=(\eta_0,\ldots,\eta_{n-1})\in \mathbb{A}^n$ the fibre $P_\eta=\psi^{-1}(\mathbf{\eta})$ is an irreducible complete intersection whose defining ideal in $\k[\L]$ is generated by $\psi_0-\eta_0,\ldots,\psi_{n-1}-\eta_{n-1}$. Similar to the classical case (investigated by Kostant in [\cite{K63}]) each fibre $P_\eta$ contains a unique open $G$-orbit, $P_\eta^\circ$, but unlike [\cite{K63}] the complement $P_\eta\setminus P_\eta^\circ$ has codimension $1$ in $P_\eta$ for any $\eta\in\mathbb{A}^n$. To be more precise, there is a semiinvariant $\Delta\in \k[\L]$ for $G$ corresponding to a nontrivial character $\chi\colon\,G\to \k^\times$ such that $P_\eta\setminus
P_\eta^\circ$ coincides with the zero locus of $\Delta$ in $P_\eta$; see [\cite{P91}, Lemmas~12 and ~15] for detail.

In [\cite{P91}], the semiinvariant $\Delta$ is constructed as follows:
if $x^{\bf p-1}:=x_1^{p-1}\cdots x_n^{p-1}$, the monomial of top degree in $\OO_n$,
then $$\Delta(D):=\big(D^{p^n-1}(x^{\bf p-1})\big)(0)\qquad\quad (\forall\,D\in\L)$$
where the notation $f(0)$ for $f\in\OO_n$ is explained in Subsection~\ref{ss1}. It is straightforward to see that $\Delta(g(D))=\chi(g)\Delta(D)$ for any $g\in G$, where $\chi=\chi_0^{p-1}$ and $\chi_0$ is the  rational character of $G$ which takes value $\lambda^{-n}$ on the
automorphism  $x_i\mapsto \lambda x_i$, $1\le i\le n$, of $\OO_n$. Our goal in this subsection is to express $\Delta$ and the $\psi_i$'s in a more traditional way inspired
by  the classical work of Dickson [\cite{Di11}].

The $n$th wedge product $\wedge^n\L$ carries a natural $G$-module structure and the action of the torus $\lambda(\k^\times)\subset G$ introduced in Subsection~\ref{ss1} turns it into a graded vector space:
$$\wedge^n\L=\textstyle{\bigoplus}_{i\ge -n}(\wedge^n\L)_i, \qquad\ \L_{-n}=\wedge^n(\L_{-1})=\k(\partial_1\wedge\cdots\partial_n).$$
Since the $p$-mapping $\pi\colon\L\to \L,\,D\mapsto D^p,$ is a morphism given by a collection of homogeneous polynomial functions of degree $p$ of $\L$, there is a homogeneous polynomial function $\Delta_0\in\k[\L]$ of degree $1+p+\cdots+p^{n-1}=(p^n-1)/(p-1)$ such that $$D\wedge D^p\cdots\wedge D^{p^{n-1}}\in\,\,\Delta_0(D)(\partial_1\wedge\cdots\wedge\partial_n)+\textstyle{\bigoplus}_{i>-n}(\wedge^n\L)_i\qquad\ (\forall\,D\in\L).
$$
Since the standard maximal subalgebra $\L_{(0)}$ is $G$-stable, so is  $\textstyle{\bigoplus}_{i>-n}(\wedge^n\L)_i$, a subspace of codimension $1$ in $\wedge^n\L$.
Since ${\rm span}\{x_1,\ldots,x_n\}\cong \L_{-1}^*$ as $G_0$-modules, it is now  routine to check that $$\Delta_0(g(D))=
\chi_0(g)\Delta_0(D)\quad\  \mbox{for all }
g\in G \mbox{ and } D\in\L.$$
By the same reasoning, for every $1\le i\le n$ there exists a homogeneous polynomial function
$\Delta_i\in\k[\L]$ such that
$$D\wedge\cdots \wedge D^{p^n}\wedge D^{p^{i}}\wedge\cdots\wedge D^{p^{n-1}}\in\,\,\Delta_i(D)(\partial_1\wedge\cdots\wedge\partial_n)+\textstyle{\bigoplus}_{i>-n}(\wedge^n\L)_i
$$
for all $D\in \L$. Obviously, $\deg \Delta_i=(p^n-p^{i-1})+\deg \Delta_0$.
\begin{prop}\label{di}
We have that $\Delta_0^{p-1}=(-1)^n\Delta$ and $\psi_{i-1}=-\frac{\Delta_i}{\Delta_0}$ for $1\le i\le n$.
\end{prop}
\begin{proof}
Given $\mathbf{\lambda}=(\lambda_1,\ldots,\lambda_n)\in\mathbb{A}^n$
we set
\begin{equation*}
\mathcal{D}_\lambda:=(1+\lambda_1x_1^{p-1})\partial_1
+x_1^{p-1}(1+\lambda_2x_2^{p-1})\partial_2+\cdots+
x_1^{p-1}\cdots x_{n-1}^{p-1}(1+\lambda_nx_n^{p-1})\partial_n
\end{equation*}
and define $\mathcal{Y}:=\{\mathcal{D}_\lambda\,|\,\,
\mathbf{\lambda}\in\mathbb{A}^n\}$. It is proved in
[\cite{P91}, \S~3] that the $G$-saturation of $\mathcal{Y}$ coincides with the principal Zariski open subset $\{D\in\L\,|\,\,\Delta(D)\ne 0\}$ of $\L$.
It is also immediate from the discussion in {\it loc.\,cit.} that
$$\mathcal{D}_\lambda^{p^i}-(-1)^i\partial_i\,\in\,
\k\partial_1\oplus\cdots\oplus\k\partial_{i-1}\oplus
\m^{p-1}\cdot\L\qquad\quad(0\le i\le n-1)$$ and
$\mathcal{D}_\lambda^{p^n-1}
(x^{\bf p-1})-(-1)^n\in\m^{p-1}\cdot\L$ for all $\lambda\in\mathbb{A}^n$.
This implies that
$$\mathcal{D}_\lambda\wedge \mathcal{D}_\lambda^p\cdots\wedge \mathcal{D}_\lambda^{p^{n-1}}\in\,\,(-1)^{n(n-1)/2}(\partial_1\wedge\cdots\wedge\partial_n)+\textstyle{\bigoplus}_{i>-n}(\wedge^n\L)_i
$$ and
$\Delta(\mathcal{D}_\lambda)=(-1)^n$ for all $\mathbf{\lambda}\in\mathbb{A}^n$. As a consequence,
$\Delta_0^{p-1}(y)=(-1)^n\Delta(y)=1$ for all $y\in\mathcal{Y}$. Since both $\Delta_0^{p-1}$ and
$(-1)^n\Delta$ are semiinvariants for $G$ associated with
the same character $\chi$, this yields that
$\Delta_0^{p-1}$ and $(-1)^n\Delta$ agree on $G\cdot\mathcal{Y}$. Since the latter is Zariski open in $\L$ we now obtain that
$\Delta_0^{p-1}=(-1)^n\Delta$.

Since $D^{p^n}=-\sum_{i=1}^{n-1}\psi_i(D)D^{p^i}$,  the definition of $\Delta_i$ in conjunction with standard properties of wedge products yields that
$\Delta_i(D)=-\psi_i(D)\Delta_0(D)$ for all $D\in\L$.
This completes the proof.
\end{proof}
Let $\t_0$ denote the $\k$-span of all $(1+x_i)\partial_i$ with $1\le i\le n$, a maximal toral subalgebra of $\L$. It is well known that the
$G$-saturation of $\t_0$ is Zariski dense in
$\L$. By [\cite{P91}], the nomaliser $N_G(\t_0)$ acts faithfully on the $\mathbb{F}_p$-space $(\t_0^{\rm tor})^*$ and is isomorphic to $\mathrm{GL}_n(\mathbb{F}_p)$. Furthermore,
the natural restriction map $\k[\L]\to \k[\t_0]$
induces an isomorphism of invariant rings
$j\colon\,\k[\L]^G\stackrel{\sim}{\longrightarrow}\,
S(\t_0^*)^{\mathrm{GL}_n(\mathbb{F}_p)}$.

For the reader's convenience we recall the original definition of Dickson invariants. Given $(\xi_1,\ldots,\xi_n)\in\mathbb{A}^n$ put
$$M_0(\xi_1,\ldots,\xi_n):=\,
\left[\begin{array}{cccc}\xi_1&\xi_2&\cdots&\xi_n\\
\xi_1^p&\xi_2^p&\cdots&\xi_n^{p}\\
\vdots&\vdots&\cdots&\vdots\\
\xi_1^{p^{n-1}}&\xi_2^{p^{n-1}}&\cdots&\xi_n^{p^{n-1}}
\end{array}\right]
$$ and let $M_i(\xi_i,\ldots,\xi_n)$ be the $n\times n$ matrix obtained by replacing the $i$th row of $M_0(\xi_1,\ldots,\xi_n)$ by $\big(\xi_1^{p^n},\ldots,\xi_n^{p^n}\big)$. For
each $0\le i\le n$ define $\varphi_i\in \k[\t_0]$ by setting
$$\varphi_i(D)
=\det M_i(\xi_1,\ldots,\xi_n)
\qquad\quad \big(\forall\, D=\textstyle{\sum}_{i=1}^n\xi_i(1+x_i)
\partial_i\in\t_0\big).$$ Thanks to [\cite{Di11}, p.~76] we know that $\bar{\psi}_{i-1}:=-\frac{\varphi_i}{\varphi_0}\in S(\t_0^*)$ for $0\le i\le n-1$ and $$S(\t_0^*)^{\mathrm{GL}_n(\mathbb{F}_p)}=\,
\k[\bar{\psi}_0,\ldots,\bar{\psi}_{n-1}].$$

Since each $(1+x_i)\partial_i$ is a toral element of $\t_0$, we have that $$D^{p^k}=\big(\textstyle{\sum}_{i=1}^n\xi_i(1+x_i)
\partial_i\big)^{p^k}=\sum_{i=1}^n\xi_i^{p^k}(1+x_i)\partial_i,\qquad\ k\ge 0.$$ This shows that
$\Delta_i(D)=\det M_i(\xi_1,\ldots,\xi_n)$ for all
$i$. As a result, the restriction of $\Delta_0$ to $\t_0$ identifies with $\varphi_0$, the Dickson semiinvariant for $\mathrm{GL}_n(\mathbb{F}_p)$. In view of Proposition~\ref{di} this means that each $\bar{\psi}_i$ can be obtained by restricting $\psi_{i}$ to $\t_0$.
\begin{rem}
We mention for completeness that $\Delta_0$ is irreducible in $\k[\L]$. Indeed, if this is not the case then $n\ge 2$ and $\Delta_0=gf$ for some $G$-semiinvariants
$f,g\in\k[\L]$ of positive degree  (this is because the connected group $G$ must preserve the line spanned by each prime divisor
of $\Delta_0$). Let $\bar{f}$ denote the restriction of $f$ to $\t_0$, a nonzero homogeneous semiinvariant for $N_G(\t_0)\cong \mathrm{GL}_n(\mathbb{F}_p)$. As $p>2$, the derived subgroup of $N_G(\t_0)$ is isomorphic to
$\mathrm{SL}_n(\mathbb{F}_p)$. But then $\bar{f}\in \k[\t_0]^{\mathrm{SL}_n(\mathbb{F}_p)}$. Since $\deg\bar{f}=\deg f<(p^n-1)/(p-1)$, this contradicts [\cite{St87}, Theorem~C]. So $\Delta_0$ is irreducible in $\k[\L]$ as claimed.
\end{rem}
\begin{rem}
The varieties  $\{x\in P_{\mathbf{\eta}}\,|\,\,\Delta_0(x)=0\}$ may be reducible fore some $\mathbf{\eta}\in\mathbb{A}^n$. In fact, this happens already when $n=1$. In this case $\Delta_0$ is a linear function on $\L$ which vanishes on $\L_{(0)}$. Applying Theorem~\ref{t1} one observes that the fibre
$P_{-1}=\psi_0^{-1}(-1)$ is a smooth hypersurface in $\L$ consisting of all {\it nonzero} toral elements of $\L$. Then $P_{-1}^{\,\circ}$ is the $G$-orbit of $(1+x_1)\partial_1$
and the zero locus of $\Delta_0$ in $P_{-1}$ coincides with $P_{-1}\setminus P_{-1}^{\,\circ}=\,\mathbb{F}_p^\times (x_1\partial_1)+\L_{(1)}$ which has $p-1$ irreducible components. Each irreducible component is a Zariski closed $G$-orbit in $\L$.
\end{rem}
\begin{rem}
The Lie algebra $\L$ operates on $\L^*$ via the
coadjoint representation and this gives rise to the action of $\L$ on $\k[\L]\cong S(\L^*)$ as derivations.
Since all derivations annihilate $p$th powers, the subring $\k[\L]^{(p)}:=\{f^p\,|\,\,f\in \k[\L]\}$ lies in the invariant algebra $\k[\L]^\L$. Although this is not directly related to the present work, we mention that $\psi_0,\ldots,\psi_{n-1}\in\k[\L]^\L$ and $\k[\L]^\L=\k[\L]^{(p)}[\psi_0,\ldots,\psi_{n-1}].$ Moreover, $\k[\L]^\L$ is a free module of rank $p^n$ over $\k[\L]^{(p)}$. This is proved by Skryabin; see [\cite{Sk02}, Corollary~5.6].
\end{rem}
\subsection{}\label{4.2} Let ${\rm Sing}(P_{\mathbf{\eta}})$ denote the subvariety of all singular points of $P_{\mathbf{\eta}}$. By [\cite{P91}, Theorem~2], the special fibre $P_{\bf 0}$ of $\psi$ coincides with the nilpotent cone $\NN(\L)$ and ${\rm Sing}(P_{\bf 0})=P_{\bf 0}\setminus P^{\,\circ}_{\bf 0}$ coincides with the zero locus of $\Delta$ in $P_{\bf 0}$. As the latter has codimension $1$ in $P_{\bf 0}$, the variety $P_{\bf 0}$ is not normal. This is, of course, in sharp contrast with a well known result of Kostant [\cite{K63}] (valid for all reductive Lie algebras over fields of characteristic $0$). Our final result in this paper shows that fibre of $\psi$ is normal if and only if it is smooth.
\begin{theorem}\label{t3}
Let $\eta\in\mathbb{A}^n$. The the following hold:
\begin{itemize}
\item[(i)\,] The fibre $P_{\mathbf{\eta}}$ is smooth if and only if it consists of regular semisimple elements of $\L$. The latter happens if and only if  $\psi_0(D)\ne 0$ for all $D\in P_{\mathbf{\eta}}$.

\smallskip

\item[(ii)\,] The fibre $P_\mathbf{\eta}$ is normal if and only if it is smooth.
\end{itemize}
\end{theorem}
\begin{proof}
(a) Let $Z=\{x\in\L\,|\,\,\psi_0(x)=0\}$. It follows from Theorem~\ref{t1} that $P_{\mathbf{\eta}}$ is a smooth variety if and only if $P_{\mathbf{\eta}}\subset \L_{\rm reg}$, whilst [\cite{P91}, Theorem~3(iv)] implies that the inclusion $P_{\mathbf{\eta}}\subset \L_{\rm reg}$ takes place if and only  $P_{\mathbf{\eta}}\cap \t_n\subset\L_{\rm reg}$. On the other hand, Theorem~\ref{t2} in conjunction with Lemma~\ref{l2} shows that
$P_{\mathbf{\eta}}\cap \t_n\subset\L_{\rm reg}$ if and only if
$\psi_0(D)\ne 0$ for all $D\in P_{\mathbf{\eta}}\cap\t_n$. Since the set $P_{\mathbf{\eta}}\cap Z$ is Zariski closed and $G$-stable, the latter occurs if and only if $P_{\mathbf{\eta}}\cap Z=\emptyset$. Statement (ii) follows.

\smallskip

\noindent
(b) If $\eta_0\ne 0$ then $P_\eta$ is smooth and hence normal; see [\cite{Sh94}, 
Ch.~2,  \S\,5, Theorem~1]. So suppose from now that $\eta_0=0$.
The we may also assume without loss that $n\ge 2$. Given $D\in \L$ we let $I_D$ stand for the unique maximal $D$-invariant ideal of $\OO_n$. It is immediate from Block's theorem [\cite{Bl69}] that  $\L/I_D \cong \OO_k$ for some $k\in\{0,1,\ldots, n\}$. By [\cite{Sk14}, Lemma~1.1(i)], each $I_D$ is $G$-conjugate to one of the ideals $I_k:=\langle x_{k+1},\ldots,x_n\rangle.$ 

Following [\cite{Sk14}] we let $\mathcal{D}_{n-1}$ denote the set of all
$D=\sum_{i=1}^nf_i\partial_i$ such that $f_n\in I_{n-1}=x_n\OO_n$
and $$f_i=x_1^{p-1}\cdots
x_{i-1}^{p-1}+x_1^{p-1}\cdots x_{n-1}^{p-1}\cdot g_i,\ \qquad 1\le i\le n-1,$$ where $g_i$ lies in the {\it subalgebra} of $\OO_n$ generated by $x_n$.
Note that $I_D=I_{n-1}$ for every $D\in\mathcal{D}_{n-1}$.

Given an endomorphism $x$ of a finite dimensional vector space $V$ over $\k$ we write $\chi_x(t)$ and $m_x(t)$ for the characteristic and minimal polynomial of $x$, respectively. Pick any $D=\sum_{i=1}^nf_i\partial_i\in\mathcal{D}_{n-1}$ with the $f_i$'s as above and let $\pi\colon\, \n_{\L}(I_{n-1})\to \Der(\OO_n/I_{n-1})$ be the canonical homomorphism. As explained in the proof of [\cite{Sk14}, Lemma~2.2], $\ker \pi=I_{n-1}\L$ and identifying $\OO_n/I_{n-1}$ with $\OO_{n-1}$ we get
$$\pi(D)=\,\textstyle{\sum}_{i=1}^{n-1}x_1^{p-1}\cdots x_{i-1}^{p-1}\partial_i+x_1^{p-1}\cdots 
x_{n-1}^{p-1}
\cdot\textstyle{\sum}_{i=1}^{n-1}\,g_{i}(0) \partial_i. $$

We now let $\mathcal{D}_{n-1}(\eta)$ be the subset of $\mathcal{D}_{n-1}$ consisting of all $D=\sum_{i=1}^nf_i\partial_i$ for which  $g_i(0)=(-1)^{n-i-1}\eta_i^{1/p}$ 
for $1\le i\le n-1$ and $D^{p^{n-1}}+\sum_{i=1}^{n-1}
\eta_i^{1/p}D^{p^{i-1}}\in\mathcal{N}(\L)$.
If $D\in\mathcal{D}_{n-1}(\eta)$ then [\cite{Sk14}, Lemma~1.4] shows that 
$$\chi_{\pi(D)}(t)=m_{\pi(D)}(t)=t^{p^{n-1}}-
\textstyle{\sum}_{i=1}^{n-1}(-1)^{n-i}g_i(0)\cdot t^{p^{i-1}}=t^{p^{n-1}}+
\textstyle{\sum}_{i=1}^{n-1}\,\eta_{i}^{1/p}t^{p^{i-1}}$$ forcing $\chi_{\pi(D)}(t)^p=t^{p^n}+\sum_{i=0}^{n-1}\eta_iD^{p^i}$ (here we use our assumption that $\eta_0=0$). Since $D\in\mathcal{D}_{n-1}(\eta)$ we also have that
$\chi_{\pi(D)}(D)\in\mathcal{N}(\L)$. This implies that $m_D(t)$ divides $\chi_{\pi(D)}(t)^{p^l}$ for some $l\in\Z_{\ge 0}$.
Since $m_D(t)$ and $\chi_D(t)$ have the same set of roots, this yields that any root of $\chi_D(t)$ is a root of $\chi_{\pi(D)}(t)$.
On the other hand, $\chi_D(\pi(D))=\pi(\chi_{D}(D))=0$ because $\pi$ is a homomorphism of restricted Lie algebras. As $\chi_{\pi(D)}(t)=m_{\pi(D)}(t)$, we deduce that $\chi_D(t)-\chi_{\pi(D)}(t)^p=\alpha
\chi_{\pi(D)}(t)$ for some $\alpha\in\k$, that is,
$$\chi_D(t)=\chi_{\pi(D)}(t)\cdot (\chi_D(t)^{p-1}+\alpha).$$
If $\alpha\ne 0$ then  $\chi_{\pi(D)}(\beta)\ne 0$ for some root $\beta$ of $\chi_{D}(t)$. Since this  possibility is ruled out by our preceding remark, it must be that $\chi_D(t)=\chi_{\pi(D)}(t)^p=t^{p^n}+\sum_{i=0}^{n-1}\eta_it^{p^i}.$ So $D\in P_\eta$, forcing $\mathcal{D}_{n-1}(\eta)\subset P_\eta$.

\smallskip

\noindent
(c) We claim that $\mathcal{D}_{n-1}(\eta)$ is a Zariski closed subset of $\L$ isomorphic to an affine space of dimension $(p-1)(n-1)+(p^n-p^{n-1}-1)$. In 
order see this we first define
$$D_\eta:=\partial_1+x_1^{p-1}\partial_2+\cdots +
x_1^{p-1}\cdots x_{n-2}^{p-1}\partial_{n-1}+
\textstyle{\sum}_{i=1}^{n-1}(-1)^{n-i-1}\eta_i^{1/p}x_1^{p-1}\cdots x_{n-1}^{p-1}
\partial_i,$$ 
put $\bar{\chi}(t):=t^{p^{n-1}}+\sum_{i=1}^{n-1}\eta_{i}^{1/p}
t^{p^{i-1}}$, and let
$D=\sum_{i=1}^nf_i\partial_i$ be an arbitrary element of $\mathcal{D}_{n-1}(\eta)$ and let the $g_i$'s with $1\le i\le n-1$ be as in part~(b).  Write $I_{n-1}'$ for the linear span of all monomials $x_1^{a_1}\cdots x_n^{a_n}$ with $0\le a_i\le p-1$ such that
$a_n\ge 1$ and $(a_1\ldots,a_n)\ne (p-1,\ldots, p-1,1)$.
Then   $f_n= c x_1^{p-1}\cdots x_{n-1}^{p-1}x_n+f_n'$ for some $c\in\k$ and some $f_n'\in I_{n-1}'$. Since $g_i(0)=(-1)^{n-i-1}\eta_i^{1/p}$ for $1\le i\le n-1$ we have that
$$D= D_\eta+D_1+x_1^{p-1}\cdots x_{n-1}^{p-1} D_2$$ where
$D_1=f_n'\partial_n$ and $D_2=c(x_n\partial_n)+
\textstyle{\sum}_{i=1}^{n-1}\,(g_i-g_i(0))\partial_i.$  It should be stressed here that $D_\eta$ is independent of the choice of $D\in \mathcal{D}_{n-1}(\eta)$ and $\bar{\chi}(D_\eta)=0$ by [\cite{Sk14}, Lemma~1.4]. Since $I_{n-1}\partial_n$ is a restricted Lie subalgebra of $\L$ normalised by $D_\eta$, Jacobson's formula entails that $\bar{\chi}(D_\eta+D_1)\in I_{n-1}\partial_n$.

Let $J$ denote the ideal of $\OO_n$ generated by $x_1,\ldots, x_{n-1}$. Using Jacobson's formula and induction on $k$ one observes that
$$
D^{p^k}\equiv (D_\eta+D_1)^{p^k}+\big(\ad(D_\eta+D_1)^{p^k-1}\big)
(x_1^{p-1}\cdots x_{n-1}^{p-1}D_2)\ \,\big({\rm mod}\, J^{(n-k-1)(p-1)+1}\L\big)
$$ for $0\le k\le n-1$ (see the proof of Lemma~1.2 in [\cite{Sk14}] for a similar argument). In view of [\cite{Sk14}, Lemma~1.3] this yields
$$
\bar{\chi}(D)\equiv\bar{\chi}(D_\eta+D_1)+
(-1)^{n-1}D_2\ \,\big({\rm mod}\,J\L\big).
$$
Since $\bar{\chi}(D)\in\ker\pi=I_{n-1}\L$ by part~(b) and both $\bar{\chi}(D_\eta+D_1)$ and $D_2$ are in $I_{n-1}\L$ by our earlier remarks, we have that
$$\bar{\chi}(D)-\bar{\chi}(D_\eta+D_1)-(-1)^{n-1}D_2\in J\L\cap I_{n-1}\L\subset I_{n-1}\m\L.$$ 
On the other hand, $I_{n-1}\L=\k(x_n\partial_n)\oplus R$ where $R:=\sum_{i=1}^{n-1}\k(x_{n}\partial_i)\oplus I_{n-1}\m\L$ is an ideal of codimension $1$ 
in the Lie algebra $I_{n-1}\L$. Since $I_{n-1}\m\L \subset\L_{(1)}$ and
 $\sum_{i=1}^{n-1}\k(x_{n}\partial_i)$ is an abelian subalgebra of $\L$ consisting of nilpotent elements, Jacobson's formula shows that $R\subset\mathcal{N}(\L)$ coincides with the nilradical of  $I_{n-1}\L$. 
 
Let $\phi\colon I_{n-1}\L\to 
\k(x_n\partial_n)\cong (I_{n-1}\L)/R$ be the canonical homomorphism and denote by $V$ the linear span of $I_{n-1}'$ and all $x_1^{p-1}\cdots x_{n-1}^{p-1}x_n^i\partial_i$ with
$1\le i\le p-1$ and $1\le j\le n-1$. 
Clearly, $$\mathcal{D}_{n-1}(\eta)\subset D_\eta+V+\k(x_1^{p-1}\cdots x_{n-1}^{p-1}x_n\partial_n)$$ and $\dim V=(p^n-p^{n-1}-1)+(n-1)(p-1)$.
As $\bar{\chi}(D)^p=0$ by part~(b), it must be that
$(\phi\circ\bar{\chi})(D)=0$.
As $D_2=c(x_n\partial_n)+\sum_{i=1}^{n-1}(g_i-g_i(0))\partial_i$ and 
$\bar{\chi}(D_\eta+D_1)$ both lie in $\k(x_n\partial_n)+R$, the above implies that $$(\phi\circ \bar{\chi})(D_\eta+D_1)
=(-1)^n c(x_n\partial_n)\qquad\quad \big(\forall\,D\in\mathcal{D}_{n-1}(\eta)\big).$$
Let $\widetilde{V}$ be the set of all elements of the form $D_{v,t}:=D_\eta+v+t(x_1^{p-1}\cdots x_{n-1}^{p-1}x_n\partial_n)$ with $v\in V$ and $t\in\k$. Then $\widetilde{V}\subseteq \mathcal{D}_{n-1}$ and part~(b) yields $\bar{\chi}(\widetilde{V})\subset\ker\pi$. Obviously, $\widetilde{V}$ is a closed subset of $\L$ isomorphic to $V\oplus\k$. 
Because we can repeat the above argument with $D=D_{v,t}\in\widetilde{V}$ 
instead of $D\in\mathcal{D}_{n-1}(\eta)$, we have that $$(\phi\circ \bar{\chi})(D_{v,t})=(-1)^nt(x_n\partial_n)\qquad\quad \ \big(\forall\,D_{v,t}\in\widetilde{V}\big).$$
This implies that there exists a regular function $F\in \k[V]$ with the property that $\bar{\chi}(D_{v,t})\in\mathcal{N}(\L)$ if and only if $t=F(v)$. 
From this it is immediate that the canonical projection $\widetilde{V}\twoheadrightarrow V$ maps $\mathcal{D}_{n-1}(\eta)\subset\widetilde{V}$ isomorphically onto $V$. The claim follows.

\smallskip

\noindent
(d) By part~(c), $\mathcal{D}_{n-1}(\eta)$ is an irreducible Zariski closed subset of $\L$ and
\begin{equation}\label{e1}
\dim\mathcal{D}_{n-1}(\eta)= (p-1)(n-1)+(p^n-p^{n-1}-1).
\end{equation}Let $G_{n-1}$ be the subgroup of $G$ consisting of all automorphisms $\sigma$ fixing $x_n$ and preserving the ideal of $\OO_n$ generated by $x_1,\ldots, x_{n-1}$. It is straightforward to see that
\begin{equation}\label{e2}
\dim G_{n-1}=(n-1)(p^n-p).
\end{equation}
By [\cite{Sk14}, Proposition~2.3], for any $D\in \L$ with
$I_D=I_{n-1}$ there exists a unique $\sigma\in G_{n-1}$ such that $\sigma(D)\in\mathcal{D}_{n-1}$.
Set $\widetilde{\mathcal{D}}_{n-1}(\eta):=G_{n-1}\cdot \mathcal{D}_{n-1}(\eta)$. Due to (\ref{e1}) and (\ref{e2}) we have 
\begin{equation}\label{e3}
\dim \widetilde{\mathcal{D}}_{n-1}(\eta)=\dim G_{n-1}+\dim\mathcal{D}_{n-1}(\eta)=np^n-p^{n-1}-n.
\end{equation}
Let $X$ denote the $G$-saturation of $\widetilde{\mathcal{D}}_{n-1}(\eta)$. The morphism $G\times _{N_G(I_{n-1})}\widetilde{\mathcal{D}}_{n-1}(\eta)\to X$ sending any 
$(g,D)\in G \times _{N_G(I_{n-1})}\widetilde{\mathcal{D}}_{n-1}(\eta)$ to $g(D)$ is obviously surjective. If $g(D)=g'(D')$ for some $g,g'\in G$ and some $D,D'\in \widetilde{\mathcal{D}}_{n-1}(\eta)$ then $g^{-1}g'(I_{n-1})=I_{g^{-1}g'(D')}=I_D=I_{n-1}$ forcing
$g'\in gN_G(I_{n-1})$. Therefore, the morphism is bijective, so that
$$\dim X=\dim G-\dim N_G(I_{n-1})+\dim\widetilde{\mathcal{D}}_{n-1}(\eta).$$
In view of (\ref{e3}) and the equality $\dim N_G(I_{n-1})=(n-1)(p^n-1)+p^n-p^{n-1}$ this yields 
\begin{eqnarray*}\dim X&=&n(p^n-1)-(n-1)(p^n-1)-(p^n-p^{n-1})+np^n-p^{n-1}-n\\
&=&
np^n-n-1.\end{eqnarray*}

\smallskip

\noindent
(e) We claim that $X\cap \L_{\rm reg}=\emptyset$. Indeed, suppose the contrary.
Since $X$ is $G$-stable, Theorem~\ref{t2} says that
$X$ contains a regular element of the form
$$D=
\partial_1+x_1^{p-1}\partial_2+
\cdots+x_1^{p-1}\cdots
x_{r-1}^{p-1}\partial_r+\textstyle{\sum}_{i=r+1}^n\,\mu_i(x_i+\epsilon_i)
\partial_i,\qquad\ 0\le r\le n,$$ where $\epsilon_i\in\{0,1\}$ and $\mu_i\in\k$, such that
$D_s=\textstyle{\sum}_{i=r+1}^n\,\mu_i(\epsilon_i+x_i)
\partial_i$ generates an $(n-r)$-dimensional torus in $\L$. If $\epsilon_i=1$ for all $i>r$ then one checks directly that $D^{p^n-1}(x^{\bf p-1})\not\in\m$ forcing $I_D=\{0\}$.
If $\epsilon_l=\epsilon_m=0$ for some $l\ne m$ then $I_D$ contains $x_l\OO_n+x_n\OO_n$, which is impossible because $I_D$ must be $G$-conjugate to $I_{n-1}$. Therefore, no generality will be lost by assuming that $I_D=I_{n-1}$, so that $\epsilon_n=0$ and $\epsilon_i=1$ for $r<i<n$. Our discussion in part~(d) then shows that $D\in\widetilde{\mathcal{D}}(\eta)$. Hence there exists a unique $g\in G_{n-1}$ such that $g(D)\in\mathcal{D}(\eta)$, so that
$g(\bar{\chi}(D))=\bar{\chi}(g(D))\in\ker\pi=I_{n-1}\L$ by part~(c). As $g^{-1}(I_{n-1})=I_{n-1}$, this yields 
$\bar{\chi}(D)\in I_{n-1}\L$. 

Let $D'=\partial_1+x_1^{p-1}\partial_2+
\cdots+x_1^{p-1}\cdots
x_{r-1}^{p-1}\partial_r+\textstyle{\sum}_{i=r+1}^{n-1}\,\mu_i(1+x_i)
\partial_i$. Then $D=D'+\mu_n(x_n\partial_n)$ and
$[x_n\partial_n, D']=0$. Since $(x_n\partial_n)^p=x_n\partial_n$, it follows that $\bar{\chi}(D)=\bar{\chi}(D')+\bar{\chi}(\mu_n)(x_n\partial_n)$ implying $\bar{\chi}(D')\in I_{n-1}\L$. Since $D'(x_i)\in\k[x_1,\ldots,x_{n-1}]$ for $1\le i\le n-1$ and $D'(x_n)=0$, this forces $\bar{\chi}(D')=0$. It follows that $\bar{\chi}(D)=\bar{\chi}(\mu_n)(x_n\partial_n)$. As
$x_n\partial_n$ is semisimple and
$\bar{\chi}(D)=g^{-1}(\bar{\chi}(g(D)))$ is nilpotent  (by the definition of $\mathcal{D}(\eta)$) we now deduce that
$\bar{\chi}(D)=0$. But then $m_D(t)$ has degree $<p^n$ contrary to Remark~2. This contradiction proves the claim.

\smallskip

\noindent
(f)
Part~(e)
together with Theorem~\ref{t1} gives $X\subseteq {\rm Sing}(P_\eta)$. As $\dim X=(\dim P_\eta)-1$ by part~(c) we conclude that ${\rm Sing}(P_\eta)$ has codimension $1$ in $P_\eta$. Then [\cite{Sh94}, 
Ch.~2,  \S\,5, Theorem~3] shows that the variety $P_\eta$ is not normal.
\end{proof}
\begin{rem}
(1) The equivalence of (i) and (ii) in Theorem~\ref{t2} was conjectured by Hao Chang (private communication). 

\smallskip

\noindent
(2) It would be interesting to describe the regular elements of $\L$ in the case where ${\rm char}(\k)=2$. Indeed, this
is a meeting point of the theory of restricted
Lie algebras and the theory of complex Lie superalgebras: if $p=2$ then $\OO_n$ is isomorphic to an exterior algebra and the Lie algebra $\L=\Der(\OO_n)$ can be obtained by reduction modulo $2$ from a suitable $\Z$-form of the finite dimensional complex Lie superalgebra of type $W_n$.

\smallskip

\noindent
(3) It would also be interesting to obtain a description of regular elements in the finite dimensional restricted Lie algebras of type $S_n$, $H_{2n}$ and $K_{2n+1}$ similar to that given in  Theorem~\ref{t2}.
\end{rem}

\end{document}